\documentclass[11pt]{article}

\usepackage[utf8]{inputenc}
\usepackage[T1]{fontenc}
\usepackage{amsmath, amsthm, amssymb}
\usepackage{mathrsfs}
\usepackage{graphicx}
\usepackage{hyperref}
\usepackage{enumitem}
\usepackage{geometry}
\newtheorem{theorem}{Theorem}[section]
\newtheorem{lem}[theorem]{Lemma}
\newtheorem{proposition}[theorem]{Proposition}
\newtheorem{corollary}[theorem]{Corollary}

\theoremstyle{definition}
\newtheorem{defn}[theorem]{Definition}
\newtheorem{rem}[theorem]{Remark}
\newtheorem{example}[theorem]{Example}

\title{\textbf{ On the Cartan torsion of Minkowskian product of Finsler manifolds}}
\author{Ranadip Gangopadhyay\thanks{Department of Mathematics, VIGNAN'S Foundation for Science, Technology and Research (Deemed to be University),Off Campus Hyderabad, Hyderabad-501512, Telangana, India, gangulyranadip@gmail.com}}
\date{}

\begin{document}
	
	\maketitle
	
	\begin{abstract}
		In this paper, we investigate the Cartan torsion and mean Cartan torsion of the Minkowskian product of two Finsler metrics. We provide explicit expressions and show that the torsion splits if and only if the Minkowskian product is Euclidean. Additionally, a condition is given for the boundedness of the mean Cartan torsion norm in the Euclidean case.
	\end{abstract}
	
\section{Introduction}

Among the fundamental tensors in Finsler geometry, the Cartan torsion plays a crucial role in measuring the deviation of a Finsler manifold from being Riemannian. A deep understanding of the Cartan torsion is therefore essential for the study of various geometric properties of Finsler manifolds. In this context, the construction of new Finsler manifolds from existing ones provides an important technique for exploring their geometric behavior. One such construction is the Minkowskian product of Finsler manifolds, introduced by Okada \cite{OKD}, which combines two Finsler metrics into a product structure on the Cartesian product of the underlying manifolds. Recently, Li et al. \cite{LHZZ} investigated dually flat and projectively flat properties of such Minkowskian products and prove that a Minkowskian product Finsler metric is dually and projectively flat if and only if it is Minkowskian. In \cite{HLBN}, the authors establish that the spray coefficients split, and the product manifold inherits Berwald-type or Landsberg-type properties exactly when the factor manifolds do. In \cite{GGPT}, authors prove that when a Minkowskian product Finsler manifold is Einstein, then either the manifold is Ricci-flat or each component manifold is itself Einstein with the same scalar curvature function. Recently, Tian et al. studied some Non-Riemannian curvature properties of Minkowskian product Finsler manifolds \cite{THLE}.

Motivated by these developments, this paper explicitly computes the Cartan and mean Cartan torsion components of Minkowskian product Finsler manifolds and establishes necessary and sufficient conditions for their splitting.

The remainder of the paper is organized as follows. Section 2 provides essential definitions and notations from Finsler geometry, including the Cartan torsion and other associated tensors. It also introduces the concept of the Minkowskian product of Finsler metrics along with its fundamental metric tensor. Section 3 focuses on the analysis of the Cartan torsion. In Proposition~\ref{p3.1}, we explicitly compute the Cartan torsion of Minkowskian product Finsler manifolds. Furthermore, Theorem~\ref{T3.1} establishes that the Cartan torsion splits if and only if the Minkowskian product is  Euclidean. Section 4 investigates the mean Cartan torsion and presents results analogous to those in Section 3.  In Section 5, we examine the conditions under which the Minkowskian product of two $C2$-like Finsler manifolds remains $C2$-like.

\section{Preliminaries}
\begin{defn}
	\cite{SSZ} \textnormal{A Finsler metric on $M$ is a function $F:TM \to [0,\infty)$ satisfying the following conditions:
		\begin{itemize}
			\item[(i)] $F$ is smooth on $TM_{0}$,
			\item[(ii)] $F$ is a positively 1-homogeneous on the fibers of tangent bundle $TM$,
			\item[(iii)]The Hessian of $\frac{F^2}{2}$ with elements $g_{ij}=\frac{1}{2}\frac{\partial ^2F^2}{\partial y^i \partial y^j}$ is positive definite on $TM_0$.
		\end{itemize}
		The pair $(M,F)$ is called a Finsler space and $g_{ij}$ is called the fundamental tensor.
	}
\end{defn}
\begin{defn}
	\textnormal{Let $\left( M,F \right) $ be a Finsler space. For a vector $y \in T_xM \setminus \left\lbrace  0\right\rbrace $, let }
	$$C_y(u,v,w) := \frac{1}{4} \frac{\partial ^3}{\partial s \partial t \partial r}[F^2(y+su+tv+rw)]_{s=t=r=0},$$
	\textnormal{where $u,v,w \in T_xM$. Each $C_y$ is a symmetric trilinear form on $T_xM$. We call the family $C:= \left\lbrace C_y : y \in T_xM \setminus \left\lbrace  0\right\rbrace\right\rbrace  $ the Cartan torsion. Define the components of Cartan torsion $C$ by}
	$$C_{ijk} = \frac{1}{4}[F^2]_{y^iy^jy^k} = \frac{1}{2}\frac{\partial}{\partial y^k}(g_{ij}),$$
	\textnormal{and the mean value of the Cartan torsion by    
		$$I_y(u) := g^{ij}(y)C_y(u,\partial_i,\partial_j) , u \in T_xM$$
		We call the family $I := \left\lbrace I_y |~ y \in T_xM \setminus \left\lbrace  0\right\rbrace \right\rbrace $ the mean Cartan torsion at $x \in M$, where $I_i = g^{jk}C_{ijk}$. }  
	\end{defn}
	\begin{defn}
		\textnormal{   For a Minkowskian space \( (V, F) \), define the norm of \( \mathbf{I} \) and \( \mathbf{C} \) as
			$$
			\| \mathbf{I} \| := \sup_{y,u \in V \setminus \{0\}} 
			\frac{F(y) | I_y(u) |}{\sqrt{g_y(u,u)}},
			$$
			$$
			\| \mathbf{C} \| := \sup_{y,u,v,w \in V \setminus \{0\}} 
			\frac{F(y) | C_y(u,v,w) |}{\sqrt{g_y(u,u) g_y(v,v) g_y(w,w)}}.
			$$}
		\end{defn}
		
		The Cartan torsion provides a measure of the deviation of a Finsler metric from being Riemannian.  
		In particular, when \( C = 0 \), the Finsler metric coincides with a Riemannian metric.  
		In general, computations involving the full Cartan torsion tensor can be quite intricate and technically demanding.  
		However, by imposing certain special forms or symmetries on the Cartan torsion, one can often uncover significant and elegant geometric properties of the underlying Finsler space.  
		To facilitate such analysis, M.~Matsumoto systematically investigated several classes of special Finsler spaces~\cite{MM}, including \( C_2 \)-like, \( C \)-reducible, and semi-\( C \)-reducible spaces, among others.
		\begin{defn}
			\textnormal{
				A Finsler metric \( F \) is said to be \emph{C2-like} if the components of its Cartan torsion \( C_{ijk} \) are expressed as
				$$C_{ijk} =  \frac{1}{I^2} I_i I_j I_k,$$
				Here, \( I_i \) is the mean Cartan torsion and $| I \|^2 := I_i I^i = g^{ij} I_i I_j.$ The quantity \( I^2 \) is defined as \( I^2 = I_i I^i =g^{ij}I_iI_j\).}
			\end{defn}
			It should be noted that every two-dimensional Finsler space is C2-like.
			\begin{lem}\cite{XCZH}
				\textnormal{For a Minkowskian space $(V, F)$,  the norm of \(I\) and \(C\) is defined by
					$$
					\|\mathbf{C}\| := \sup_{y, u, v, w \in V \setminus \{0\}} \frac{F(y) \left| C_y(u, v, w) \right|}{\sqrt{g_y(u, u) g_y(v, v) g_y(w, w)}}.
					$$
					$$
					\|\mathbf{I}\| := \sup_{y, u \in V \setminus \{0\}} \frac{F(y) \left| I_y(u) \right|}{\sqrt{g_y(u, u)}}, 
					$$
				}

			\end{lem}
			In local coordinates,
			\begin{equation}
			\| C \|^2 := C_{ijk} C^{ijk} = g^{ip} g^{jq} g^{kr} C_{ijk} C_{pqr}.   
			\end{equation}
			and
			\begin{equation}
			I_{\alpha}=g^{\beta\gamma}C_{\alpha\beta\gamma}.   
			\end{equation}
			
			Here \( g^{ij} \) denotes the inverse of the fundamental tensor \( g_{ij} \), and indices are raised using \( g^{ij} \).
			
			\begin{rem}\cite{ChSh}
				\textnormal{A two dimensional Finsler space is always C2-like as well as C-reducible where as, a three dimensional Finsler space is always semi C-reducible.}
			\end{rem}
			
			\begin{lem}\cite{XCZH}
				\textnormal{
					Let \( F = \alpha + \beta \) be a Randers norm on an \( n \)-dimensional vector space \( V \). Then
					\[
					\|\mathbf{I}\| = \frac{n+1}{\sqrt{2}} \sqrt{1 - \sqrt{1 - b^2}} < \frac{n+1}{\sqrt{2}}, 
					\]
					\[
					\|\mathbf{C}\| \leq \frac{3}{\sqrt{2}} \sqrt{1 - \sqrt{1 - b^2}} < \frac{3}{\sqrt{2}}, 
					\]
					\textit{where \( b := \|\beta\|_\alpha \).
					}
				}
			\end{lem}

			Let $(\overline{M}, \overline{F})$ and $(\widetilde{M}, \widetilde{F})$ be two Finsler manifolds with dimension $m$ and $n$, respectively, then $M=\overline{M} \times \widetilde{M} $ is a product manifold with dimensions $m+n$. and $(x^1,...,x^m)$ and $(x^{m+1},...,x^{m+n})$ be the local coordinates of $\overline{M}$ and $\widetilde{M}$, respectively, then the local coordinates on $M$ are $(x^1,...,x^{m+n})$.
			
			We assume $(x^1,...,x^m,y^1,...,y^m)$ and $(x^{m+1},...,x^{m+n},y^{m+1},...,y^{m+n})$ be the induced local coordinates on the tangent bundle $T\overline{M}$ and $T\widetilde{M}$, respectively, then the induced local coordinates on the tangent bundle $TM$ are $(x^1,...,x^m,x^{m+1},...,x^{m+n},y^1,...,y^m,y^{m+1},...,y^{m+n})$. Denote $\overline{M}' = T\overline{M} \setminus\{0\}$, $\widetilde{M}'= T\widetilde{M} \setminus \{0\}$, $M'= \overline{M}' \times \widetilde{M}' \subset T(\overline{M} \times \widetilde{M}) \setminus \{0\}$. \\
			The following, lowercase Latin indices such as $i,j,k$, etc.,will run from $1$ to $m+n$, lowercase Latin indices such as $a,b,c$, etc.,will run from $1$ to $m$, whereas lowercase Greek indices such as $\alpha,\beta,\gamma$, etc.,will run from $m+1$ to $m+n$, and the Einstein summation convention is assumed throughout this paper.\\
			Let $f:[0,\infty ) \times [0,\infty ) \to [0,\infty )$ be a continuous function such that
			\begin{itemize}
				\item[(a)] $f(s,t)=0$ if and only if $(s,t) = (0,0)$; 
				\item[(b)] $f(\lambda s,\lambda t)=\lambda f(s,t)$ for any $\lambda \in [0,\infty )$;
				\item[(c)] $f$ is smooth on $ (0,\infty ) \times (0,\infty )$;
				\item[(d)] $\frac{\partial f}{\partial s} \neq 0$, $\frac{\partial f}{\partial t} \neq 0$ for any $ (s,t)\in (0,\infty ) \times (0,\infty )$;
				\item[(e)] $\frac{\partial f}{\partial s} \frac{\partial f}{\partial t}- 2f\frac{\partial ^2 f}{\partial s \partial t } \neq 0$ for any $  (s,t)\in (0,\infty ) \times (0,\infty )$.
			\end{itemize}
			\begin{example}
				\textnormal{Let$f:[0,\infty ) \times [0,\infty ) \to [0,\infty )$ such that $f(s,t)=as+bt$, where $a,b$ are positive real numbers. Then $f$ satisfies the above five properties.In this case, the product is called Euclidean Minkowskian product.}  
			\end{example}
			\begin{example}
				\textnormal{Let \( F_1 \) and \( F_2 \) be two Finsler metrics on a manifold \( M \). The $L^p$-Minkowskian product of \( F_1 \) and \( F_2 \) is defined as:
					$$F(x, y) = \left( F_1(x, y)^p + F_2(x, y)^p \right)^{1/p}, \quad p > 1.$$
					Previous example was the particular case for $p=2$}   
				\end{example}
				\begin{defn}\cite{OKD}
					\textnormal{Let $(\overline{M}, \overline{F})$ and $(\widetilde{M}, \widetilde{F})$ be two Finsler manifolds and $f$ be continuous function satisfying the conditions $\mathbf{(a)-(e)}$. Denote $K= \overline{F}^2$, $H=\widetilde{F}^2$, the Minkowskian product Finsler manifold $(M, F)$ of $(\overline{M},  \overline{F})$ and $(\widetilde{M}, \widetilde{F})$ with respect to the product function $f$ is the product manifold $M=\overline{M} \times \widetilde{M}$ endowed with the Finsler metric $F: M' \to \mathbb{R}^+$ given by
						\begin{equation}
						F (x, y) = \sqrt{f (K (x^i, y^i),H (x^{i'}, y^{i'}))}, 
						\end{equation}
						where $(x, y) \in M'$, $(x^i, y^{i}) \in \overline{M}'$, $(x^{i'}, y^{i'}) \in \widetilde{M}'$ with $x =(x^i, x^{i'})$, $y=(y^i, y^{i'})$}.    
					\end{defn}  
					In \cite{OKD} Okada showed that $F$  defined  above is a Finsler metric on $M$.
					
					\begin{proposition}\cite{WZ}
						\textnormal{
							If $f(K, H)$ is a homogeneous function with respect to $K$ and $H$, then
							\begin{equation}
							f_{K}K + f_{H}H = f,~ f_{KK}K + f_{KH}H = 0,~ f_{HK}K + f_{HH}H = 0,~ f^2_{KH} = f_{KK}f_{HH}.
							\end{equation}
						}
					\end{proposition}
					In this paper we use the following notations:
					\begin{equation*}
						K_i=\frac{\partial K}{\partial y^i}, ~ K_{;i}=\frac{\partial K}{\partial x^i}, ~ K_{i;j}=\frac{ \partial ^2K}{ \partial y^i\partial x^j},   
					\end{equation*}
					\begin{equation*}
						H_{i'}=\frac{\partial H}{\partial y^{i'}}, ~ H_{;{i'}}=\frac{\partial H}{\partial x^{i'}}, ~ H_{i';j'}=\frac{ \partial ^2H}{ \partial y^{i'}\partial x^{j'}}.  
					\end{equation*}
					Since $K$ and $H$ are homegenous functions of degree $2$ with respect to $y$ we get the following results:
					\begin{equation}\label{4.05}
					K_{ij}y^j= K_i, ~ K_iy^i= 2K, ~ H_{i'j'}y^{j'}=H_{i'}, ~ H_{i'}y^{i'}= 2H.
					\end{equation}
					
					\begin{proposition}\cite{HLBN}
						\textnormal{
							Let $(M,F)$ be a Minkowskian product of the Finsler manifolds $(\overline{M},\overline{F})$ and $(\widetilde{M},\widetilde{F})$. Then the fundamental tensor matrix of $F$ is given by 
							\begin{equation}\label{4.1}
							g_{\alpha\beta}= \frac{1}{2}\left( \frac{\partial^2 F^2}{\partial y^{\alpha} \partial y^{\beta}}\right)= \begin{pmatrix} G_{ij} & G_{ij'} \\ G_{i'j} & G_{i'j'}, \end{pmatrix} 
							\end{equation}
							where
							\begin{equation}\label{4.2}
							\begin{split}
							G_{ij}= f_KK_{ij}+f_{KK}K_iK_j, \qquad G_{ij'}= f_{KH}K_iH_{j'},\\ G_{i'j}= f_{KH}H_{i'}K_j, \qquad  G_{i'j'}=f_HH_{i'j'}+f_{HH}H_{i'}H_{j'}.   
							\end{split}    
							\end{equation}
						}
					\end{proposition}
					\begin{proposition}\cite{HLBN}
						\textnormal{
							Let $(M,F)$ be a Minkowskian product of the Finsler manifolds $(\overline{M},\overline{F})$ and $(\widetilde{M},\widetilde{F})$. Then the inverse matrix of the fundamental tensor matrix of $F$ is given by 
							\begin{equation}\label{4.3}
							g^{\alpha\beta}= \begin{pmatrix} G^{ji} & G^{ji'} \\ G^{j'i} & G^{j'i'} \end{pmatrix} 
							\end{equation}
							where
							\begin{equation}\label{4.4}
							\begin{split}
							G^{ji}=\frac{1}{f_K}\left( K^{ji}- \frac{f_Hf_{KK}}{\Delta}y^jy^i\right), \qquad G^{ji'}= -\frac{1}{\Delta}f_{KH}y^jy^{i'},\\ G^{j' i}= -\frac{1}{\Delta}f_{KH}y^{j'}y^i, \qquad  G^{i'j'}= \frac{1}{f_H}\left( H^{j'i'}- \frac{f_Kf_{HH}}{\Delta}y^{j'}y^{i'}\right).   
							\end{split}    
							\end{equation}
							and $\Delta=f_Kf_H-2ff_{KH}$
						}
					\end{proposition}
					\section{Cartan torsion of Minkowskian product of Finsler manifolds}
					In this section, we study the Cartan torsion of Minkowskian product of Finsler manifolds $(\bar{M},\bar{F})$ and $\tilde{M}, \tilde{F}$ and denote it by $(M,F)$. We consider $f$ is the corresponding Minkowskian function. Also we assume $K=\bar{F}^2$ and $H=\tilde{F}^2$.
					
					Let $C_{\alpha\beta\gamma}$, $\bar{C}_{ijl}$, $\tilde{C}_{i'j'l'}$ be the Cartan torsion of $M$, $\bar{M}$ and $\tilde{M}$ respectively. Then
					\begin{equation}
					\bar{C}_{ijl}=\frac{1}{2}K_{ijl}, ~~ \tilde{C}_{i'j'l'}= \frac{1}{2}H_{i'j'l'}
					\end{equation}
					
					Differentiating the first equation of \eqref{4.2} with respect $y^i$ we obtain the following:
					\begin{align}
						2C_{ijl}  \nonumber
						&= \frac{\partial G_{ij}}{\partial y^l}  \\ \nonumber
						&= f_{KK} K_l K_{ij} + f_K K_{ijl} + f_{KKK} K_i K_j K_l + f_{KK} K_{il} K_j+ f_{KK} K_i K_{jl} \\ 
						&= f_K K_{ijl} + f_{KK} \bigl( K_l K_{ij} + K_{il} K_j + K_i K_{jl} \bigr) + f_{KKK} K_i K_j K_l
					\end{align}
					Similarly, differentiating the different terms of \eqref{4.2} with respect to the coefficients of $y$ we obtain the following proposition:
					\begin{proposition}\label{p3.1}
						\textnormal{
							The Cartan torsion of Minkowskian product Finsler manifold is given by for $\alpha,\beta,\gamma\in \{1,2,...,n_1\}$
							\begin{equation}\label{4.5}
							C_{\alpha\beta\gamma}=C_{ijl}=\frac{1}{2}\left[ f_{K}K_{ijl}+f_{KK}(K_iK_{jl}+K_jK_{il}+K_lK_{ij})+f_{KKK}K_iK_jK_l\right] 
							\end{equation}
							for $\alpha,\beta\in \{1,2,...,n_1\}$ and for $\gamma\in \{n_1+1,...,n_1+n_2\}$
							\begin{equation}\label{4.6}
							C_{\alpha\beta\gamma}=C_{ijl'}=\frac{1}{2}\left[ f_{KHK}K_iK_jH_{l'}+f_{KH}K_{ij}H_{l'}\right] ; 
							\end{equation}
							for $\gamma,\alpha\in \{1,2,...,n_1\}$ and for $\beta\in \{n_1+1,...,n_1+n_2\}$
							\begin{equation}\label{4.006}
							C_{\alpha\beta\gamma}= C_{ij'l}=\frac{1}{2}\left[ f_{KHK}K_iH_{j'}K_l+f_{KH}K_{il}H_{j'}\right] ;
							\end{equation}
							for $\gamma,\alpha\in \{1,2,...,n_1\}$ and for $\beta\in \{n_1+1,...,n_1+n_2\}$
							\begin{equation}\label{4.0007}
							C_{\alpha\beta\gamma}= C_{i'jl}=\frac{1}{2}\left[ f_{KHK}H_{i'}K_{j}K_l+f_{KH}K_{jl}H_{i'}\right] ;
							\end{equation}
							for $\alpha \in \{1,2,...,n_1\}$ and for $\beta, \gamma\in \{n_1+1,...,n_1+n_2\}$
							\begin{equation}\label{4.7}
							C_{\alpha\beta\gamma}=C_{ij'l'}=\frac{1}{2}\left[ f_{KHH}K_iH_{j'}H_{l'}+f_{KH}K_{i}H_{j'l'}\right] ;
							\end{equation}
							for $\gamma \in \{1,2,...,n_1\}$ and for $\alpha,\beta \in \{n_1+1,...,n_1+n_2\}$
							\begin{equation}\label{4.007}
							C_{\alpha\beta\gamma}=C_{i'j'l}=\frac{1}{2}\left[ f_{KHK}H_{i'}H_{j'}K_l+f_{KH}H_{i'j'}K_{l}\right] 
							\end{equation}
							for $\beta  \in \{1,2,...,n_1\}$ and for $\alpha, \gamma \in \{n_1+1,...,n_1+n_2\}$
							\begin{equation}\label{4.8}
							C_{\alpha\beta\gamma}=C_{i'jl'}=\frac{1}{2}\left[ f_{KHH}K_jH_{i'}H_{l'}+f_{KH}K_{l}H_{i'l'}\right]; 
							\end{equation}
							for $\gamma  \in \{1,2,...,n_1\}$ and for $\alpha, \beta   \in \{n_1+1,...,n_1+n_2\}$
							\begin{equation}\label{4.008}
							C_{\alpha\beta\gamma}= C_{i'j'l}=\frac{1}{2}\left[ f_{HK}H_{i'j'}K_l+f_{KHH}H_{i'}H_{j'}K_{l}\right] 
							\end{equation}
							for $\alpha,\beta,\gamma\in \{n_1+1,...,n_1+n_2\}$
							\begin{equation}\label{4.9}
							C_{\alpha\beta\gamma}=C_{i'j'l'}=\frac{1}{2}\left[ f_{H}H_{i'j'l'}+f_{HH}(H_{i'}H_{j'l'}+H_{j'l'}H_{i'l'}+H_{l'}H_{i'j'})+f_{HHH}H_{i'}H_{j'}H_{l'}\right] 
							\end{equation}	
						}
					\end{proposition}
					Before proving the next theorem we need the following lemma. \begin{lem}\label{lemma2}
						\textnormal{Let \( f(x, y) \) be a smooth function of two variables that is positively 1-homogeneous. If the mixed partial derivative \( f_{xy} \) vanishes identically, then \( f \) is a linear function. 
						}
					\end{lem}
					\begin{proof}
						Since \( f_{xy} = 0 \), we have that \( f_x(x,y) \) is independent of \( y \), and \( f_y(x,y) \) is independent of \( x \). Let us assume $f_x=g(x)$ and $f_y=h(y)$. Then 
						\begin{equation*}
							df=f_xdx+f_ydy=g(x)dx+h(y)dy  
						\end{equation*}
						Integrating we obtain
						\[
						f(x,y) = G(x)+H(y)+C,
						\]
						where \(G(x)=\int g(x)dx \) and \(H(y)=\int h(y)dy \).\\
						Since $f$ is homogeneous $c=0$. Hence,
						\[
						f(x,y) = G(x)+H(y),
						\]
						Again, since \( f \) is 1-homogeneous, using Euler's homogeneous equation formula we have, $x f_x + y f_y = f.$ Substituting \( f(x,y) = G(x) + H(y) \), \( f_x = G'(x) \), and \( f_y = H'(y) \), and rewriting we get,
						\begin{equation}
						x G'(x) - G(x) = -H(y) + y H'(y).   
						\end{equation}
						
						Note that the left-hand side depends only on \( x \), while the right-hand side depends only on \( y \). Therefore, both sides must equal a constant, say \( c \). That is,
						\begin{equation}
						x G'(x) - G(x) = c, \quad y H'(y) - H(y) = -c.    
						\end{equation}
						These are first order linear differential equation. It is need to be noted that as $f$ is 1-degree homogeneous  constant terms must be $0$. Therefore, solving them we obtain $G(x) = A x$ and $H(y) = B y$, where $A$ and $B$ are constants.    Hence, $f$ is a linear function.
					\end{proof}
					
					\begin{theorem}\label{T3.1}
						\textnormal{ The Cartan torsion of Minkowskian product of Finsler manifolds splits, that is, 
							\begin{equation}
							C_{\alpha \beta \gamma} =
							\begin{cases} 
							a\bar{C}_{ijl} & \text{if } (\alpha, \beta, \gamma) = (i,j,l), \\[6pt] 
							b\tilde{C}{i'j'l'}   & \text{if } (\alpha, \beta, \gamma) = (i',j',l'), \\[6pt] 
							0 & \text{otherwise }
							\end{cases}
							\end{equation}
							if and only if $f$ is linear in $H$ and $K$ that is $f= aH+bK$, where $a$ and $b$ are real constants}
						\end{theorem}
						\begin{proof}
							If $f$ is linear then all the second and higher order derivatives are zero. Hence from the above proposition results are follow immediately.
							
							The converse part is immediate from Proposition \ref{p3.1} and Lemma \ref{lemma2}.
						\end{proof}
						\subsection*{\textbf{Norm of the Cartan torsion} }
						Next we compute the norm of the Cartan torsion of the Minkowskian product metric.

						By definition of the norm of the Cartan torsion:
						\[
						\|C\|^2 = C_{\alpha \beta \gamma} C^{\alpha \beta \gamma} = g^{\alpha \nu} g^{\beta \mu} g^{\gamma \eta} C_{\alpha \beta \gamma} C_{\nu \mu \eta}.
						\]
						Let us assume
						\begin{equation}
						\|\bar{C}\|^2 = \bar{C}_{ijk} \bar{C}^{ijk}, \quad
						\tilde{C}^2 = \tilde{C}_{i'j'k'} \tilde{C}^{i'j'k'}.   
						\end{equation}
						Therefore, 
						\begin{equation}
						\|C\|^2 =  g^{ia} g^{jb} g^{kc} C_{ijk} C_{abc}+  g^{i'a'} g^{j'b'} g^{k'c'} C_{i'j'k'} C_{a'b'c'} + \quad \textnormal{other mixed terms}
						\end{equation}
						
						Using \eqref{4.4}, \eqref{4.008} and the fact $y^iC_{ijk}=0$ we obtain
						\begin{equation}
						\begin{split}
						g^{ia} g^{jb} g^{kc} C_{ijk} C_{abc}=\frac{1}{f_K^2}\|\bar{C}\|^2-\frac{f_Hf_{KK}}{\Delta}\left(\sum_{i,j,k,a,b,c}K^{bj}K^{ck}y^ay^i\right)\bar{C}_{ijk}\bar{C}_{abc}\\ -\left( \frac{f_Hf_{KK}}{\Delta}\right)^2\left( \sum_{i,j,k,a,b,c}K^{bj}y^ay^cy^iy^k\right)\bar{C}_{ijk}\bar{C}_{abc}+  \left( \frac{f_Hf_{KK}}{\Delta}\right)^3y^iy^jy^ky^ay^by^c \bar{C}_{ijk}\bar{C}_{abc}\\
						= \frac{1}{f_K^2}\|\bar{C}\|^2 \hspace{10cm}
						\end{split}
						\end{equation}
						Similarly,
						\begin{equation}
						\begin{split}
						g^{i'a'} g^{j'b'} g^{k'c'} C_{i'j'k'} C_{a'b'c'}=\frac{1}{f_H^2}\|\tilde{C}\|^2
						\end{split}
						\end{equation}
						Also by the fact $y^iC_{ijk}=0$ it can be easily obtain that all the mixed terms become zero.
						Therefore,
						\begin{equation}\label{eq3}
						\|C\|^2 =
						\frac{1}{f_K^2} \|\bar{C}\|^2 + \frac{1}{f_H^2}\|\tilde{C}^2\|.   
						\end{equation}
						This gives the following theorem:
						\begin{theorem}
							\textnormal{ The norm of Cartan torsion of the Minkowskian product $(M,F)$ of the Finsler manifolds $(\overline{M}, \overline{F})$ and $(\widetilde{M}, \widetilde{F})$ is given by }
							\begin{equation}
							\|C\|^2 =
							\frac{1}{f_K^2} \|\bar{C}\|^2 + \frac{1}{f_H^2}\|\tilde{C}^2\|.  
							\end{equation}
						\end{theorem}
						
						\begin{theorem}
							\textnormal{
								The norm of Cartan torsion of the  Minkowskian product $(M,F)$ of the Finsler manifolds $(\overline{M}, \overline{F})$ and $(\widetilde{M}, \widetilde{F})$ is bounded if and only if both $(M,F)$ and $(\overline{M}, \overline{F})$  have bounded Cartan torsion
							}    
						\end{theorem}
						\begin{corollary}
							\textnormal{
								The  Cartan torsion of  Minkowskian product of two Randers metric is bounded.
							}
						\end{corollary}
						\section{Mean Cartan torsion}
						In this section we study the mean Cartan torsion and it's norm for the Minkowskian product of Finsler manifolds. The mean cartan torsion is defined by $I_{\alpha}=g^{\beta\gamma}C_{\alpha\beta\gamma}$. Therefore, for $\alpha=i$
						\begin{equation}\label{4.10}
						I_i= g^{jl}C_{ijl}+g^{j'l}C_{ij'l}+g^{jl'}C_{ijl'}+g^{j'l'}C_{ij'l'}
						\end{equation}
						Now
						\begin{equation}
						\begin{split}
						g^{jl} C_{ijl}&=
						\frac{1}{f_K} (k^l - f_H f_{KK} y^j y^l) [f_K K_{jl} + f_{KK} (K_j K_l + K_l K_j + K_j K_l)]\\
						&= \frac{1}{f_K} [f_K K_{jl} + f_{KK} (K_j K_l + K_l K_j + K_j K_l) + f_{KKK} K_j K_l K_i - \frac{f_H f_{KK} K_j y^l}{\Delta}\\
						&+ f_{KK} (K_j K_l y^l + K_l K_j y^l + K_j K_l y^l) + f_{KKK} K_j K_l K_i K_j K_l]\\
						&= \frac{1}{f_K} [f_K K_{jl} + f_{KK} (K_i (n_i+2) + 2 f_{KKK} K_i - f_H f_{KK} \{6 K_i f_K K_j + 4K^2 f_{KKK}\}]\\
						&= \frac{1}{f_K} \left[f_K K_{jl} + K_i ((n_i+2) f_{KK} + 2K f_{KKK}) - \frac{2K K_i f_H f_{KK} (3 f_{KK} + 2K f_{KKK})}{\Delta}\right], 
						\end{split}
						\end{equation}
						
						Similarly we obtain,
						\begin{equation}
						\begin{split}
						g^{j i'} C_{i j i'}  &= -\frac{2 K_i H}{\Delta} f_{KH} \big( 2K f_{KH,K} + f_{KH} \big),\\  g^{jl'} C_{i jl'} &= -\frac{2 K_i H}{\Delta} f_{KH} 
						\big( f_{KH} + 2K f_{KH,K} \big), \quad \textnormal{and} \quad \\ g^{j'l'} C_{jl'} 
						&= \frac{1}{f_H} \left[2 f_{KHH} K_i H + \eta_2 f_{KH} K_i - \frac{f_{KHH}}{\Delta} \left(f_{KHH} K_i H^2 + 2 f_{KH} K_i H\right)\right]
						\end{split}
						\end{equation}
						Adding all the terms  we obtain
						\begin{equation}\label{4.11}
						I_i= \bar{I_i}+K_iU,
						\end{equation}\label{4.12}
						Where $\bar{I_i}=K^{jk}K_{ijk}$ is the mean Cartan torsion of $\overline{M}$ and 
						\begin{equation}\label{4.41}
						\begin{split}
						U=\frac{1}{f_K}\left[(n_1+2)f_{KK}+2Kf_{KK}-\frac{2Kf_Hf_{KK}}{\Delta}(3f_{KK}+2Kf_{KKK}) \right]\\-\frac{4H}{\Delta}f_{KH}(2Kf_{KKH}+f_{KH})\\+ \frac{1}{f_H}\left[ 2Hf_{KHH}+f_{KH}+n_2f_{KH}-\frac{2Hf_Kf_{HH}}{\Delta}(2Hf_{KHH}+f_{KH}) \right]
						\end{split}
						\end{equation}
						Similarly, for $\alpha=i'$
						\begin{equation}\label{4.13}
						I_{i'}= \tilde{I_i}+H_{i'}V,
						\end{equation}
						Where $\tilde{I_i}=H^{j'k'}H_{i'j'k'}$ is the mean Cartan torsion of $\tilde{M}$ and 
						\begin{equation}\label{4.42}
						\begin{split}
						V=\frac{1}{f_K}\left[2Kf_{KKH}+n_1f_{KH}-\frac{2Kf_Hf_{KK}}{\Delta}(2Kf_{KKH}+f_{KH}) \right]\\-\frac{4K}{\Delta}f_{KH}(2Hf_{KHH}+f_{KH})\\+ \frac{1}{f_H}\left[ (n_2+2)f_{HH}+2Hf_{HH}-\frac{2Hf_Kf_{HH}}{\Delta}(3f_{HH}+2Hf_{HHH}) \right]
						\end{split}
						\end{equation}
						Now rearranging the terms of \eqref{4.41} we can write
						\begin{equation}\label{4.06}
						\begin{split}
						U=\frac{f_{KK}}{f_K}(n_1+2) +n_2\frac{f_{KH}}{f_H}+\frac{1}{f_K}\left[ 2Kf_{KK}-\frac{2Kf_Hf_{KK}}{\Delta}(3f_{KK}+2Kf_{KKK}) \right]\\-\frac{4H}{\Delta}f_{KH}(2Kf_{KKH}+f_{KH})+  \frac{1}{f_H}\left[ 2Hf_{KHH}+f_{KH}-\frac{2Hf_Kf_{HH}}{\Delta}(2Hf_{KHH}+f_{KH}) \right]
						\end{split}
						\end{equation}
						Suppose $U=0$. Since $n_1$ and $n_2$ are arbitrary positive constants from the first two terms of the above equation we have $f_{KK}=0$ and $f_{KH}=0$. If $V=0$ by similar argument from \eqref{4.42} we get $f_{HH}=0$. Hence, $f$ is linear in $H$ and $K$. It is evident that if $f$ is linear then $U=V=0$. Hence we have the following proposition:
						\begin{proposition}
							\textnormal{
								The mean Cartan torsion of the Minkowskian product $(M,F)$ of the Finsler manifolds $(\overline{M}, \overline{F})$ and $(\widetilde{M}, \widetilde{F})$ are 
								\begin{equation}
								I_{\alpha}= \begin{cases}\overline{I}_i  \quad :  \alpha = i \\  \widetilde{I}^{i'}   \quad :  \alpha = i'\end{cases}
								\end{equation}
								if and only if $f$ is linear in $H$ and $K$.
							}
						\end{proposition}
						\subsection*{\textbf{Norm of mean Cartan torsion}}
						
						Next we find the norm of mean Cartan torsion $\|I\|^2$. Let us denote the norm of mean Cartan torsion of the Finsler manifolds $\bar{F}$ and $\tilde{F}$ by $\bar{\|I\|}$ and $\tilde{\|I\|}$ respectively. Therefore,
						\begin{equation}
						\bar{\|I\|}^2=K^{li} I_iI_l, \quad    \tilde{\|I\|}^2=H^{j'i'}I_{i'}I_{j'}
						\end{equation}
						Hence, 
						\begin{align*}
							\|I\|^2 &= g^{\alpha \beta} I_{\alpha} I_{\beta} \\
							\|I\|^2 &= g^{ij} I_i I_j + g^{ij'} I_i I_j' + g^{i'j} I_{i'} I_j + g^{i'j'} I_{i'} I_{j'} \\
							&= \frac{1}{f_K} \left( K^{li} - \frac{f_H f_{KK} y^i y^j}{\Delta} \right) (I_i + K_i U) (I_j + K_j U)\\ & - \frac{1}{\Delta} f_{KH} y^i y^j (I_i + K_i U) (I_j + K_j U) \\
							&\quad + \frac{1}{f_H} \left( H^{ji} - \frac{f_K f_{HH} y^i y^j}{\Delta} \right) (I_i + H_i V) (I_j + H_j V) \\
							&= \frac{1}{f_K} \left( \bar{I}^2 + 2U^2 K - \frac{4V^2 f_{H} f_{KK}}{\Delta} \right) - \frac{8 f_{KH} UV KH}{\Delta} \\
							&\quad + \frac{1}{f_H} \left( H^{ji} - \frac{f_K f_{HH} y^i y^j}{\Delta} \right) (\tilde{I}_i \tilde{I}_j + \tilde{I}_i H_j V + H_i V \tilde{I}_j + H_i H_j V^2) \\
							I^2 &= \frac{1}{f_K} \left( \bar{I}^2 + 2U^2 K - \frac{4V^2 f_H f_{KK}}{\Delta} \right) - \frac{8 f_{KH} UV KH}{\Delta} \\
							&\quad + \frac{1}{f_H} \left( \tilde{I}^2 + 2V^2 H - \frac{4V^2 f_K f_{HH}}{\Delta} \right)
						\end{align*}
						Now if $f$ is linear and if it is of the form $f=aH+bK$, then already we have seen that $U=0$ and $V=0$. Hence 
						\begin{equation}\label{mean}
						I^2 = \frac{1}{a} \bar{I}^2+  \frac{1}{b} \tilde{I}^2 
						\end{equation}
						Hence from \eqref{mean} we have the following theorem:
						
						\begin{theorem}
							\textnormal{
								The norm of mean Cartan torsion of the Euclidean Minkowskian product $(M,F)$ of the Finsler manifolds $(\overline{M}, \overline{F})$ and $(\widetilde{M}, \widetilde{F})$ is bounded if and only if both $(M,F)$ and $(\overline{M}, \overline{F})$  have bounded mean Cartan torsion.
							}    
						\end{theorem}
						\begin{rem}
							\textnormal{It is interesting to note that for any Minkowskian product of Finsler metrics, the norm of the Cartan torsion admits an additive decomposition in terms of the Cartan torsions of the component metrics. However, a similar decomposition for the mean Cartan torsion holds only when the Minkowskian product is Euclidean.}     
						\end{rem}
						
						\begin{corollary}
							\textnormal{
								The Mean Cartan torsion of Euclidean Minkowskian product of two Randers metric is bounded.
							}
						\end{corollary}
						
						\section{ C2-like Finsler metric Minkowskian product Finsler metrics }
						A Finsler metric \( F \) is said to be C2-like if the components of its Cartan torsion \( C_{ijk} \) are expressed as
						\begin{equation}
						C_{ijk} =  \frac{1}{I^2} I_i I_j I_k,
						\end{equation}
						\begin{theorem}\label{th5.1}
							\textnormal{Let $(M,F)$ be a linear Minkowskian product of the Finsler manifolds $(\overline{M},\overline{F})$ and $(\widetilde{M},\widetilde{F})$. If $(\overline{M},\overline{F})$ and $(\widetilde{M},\widetilde{F})$ are $C2$-like Finsler metrics. Then, $(M,F)$ is a $C2$-like Finsler metric if and only if at least one of the quotient metrics $\overline{F}$ or $\widetilde{F}$ must be Riemannian.}
							
						\end{theorem}
						\begin{proof}
							From the definition of $C2$-like Finsler metrics we get
							\begin{equation*}
								C_{ij'k}=\frac{1}{I^2}\overline{I}_i\widetilde{I}_{j'}\overline{I}_k
							\end{equation*}
							Since $f$ is linear we have $C_{ij'k}=0$. Therefore, $\frac{1}{I^2}\overline{I}_i\widetilde{I}_{j'}\overline{I}_k=0$. Hence, either $\overline{I}_i=0$, or, $\widetilde{I}_{j'}=0$. Hence by Diecke's Theorem the result follows.
						\end{proof}
						\begin{rem}
							\textnormal{ It is important to note that results analogous to Theorem \ref{th5.1} cannot be obtained for C-reducible or semi-C-reducible metrics,  due to the presence of dimension-dependent terms $n_1$ and $n_2$.}
						\end{rem}

					\end{document}